\documentclass[12pt]{amsart}
\usepackage{hyperref,amssymb}

\newtheorem{theorem}{Theorem}[section]
\newtheorem{definition}[theorem]{Definition}
\newtheorem{conjecture}[theorem]{Conjecture}
\newtheorem{proposition}[theorem]{Proposition}

\begin{document}

\title{Decomposition results for Gram matrix determinants}

\author{Teodor Banica}
\address{T.B.: Department of Mathematics, Cergy-Pontoise University, 95000 Cergy-Pontoise, France. {\tt teodor.banica@u-cergy.fr}}

\author{Stephen Curran}
\address{S.C.: Department of Mathematics, University of California, Los Angeles, CA 90095, USA. {\tt curransr@math.ucla.edu}}

\subjclass[2000]{46L54 (15A52, 33C80)}
\keywords{Orthogonal polynomial, Gram determinant}

\begin{abstract}
We study the Gram matrix determinants for the groups $S_n,O_n,B_n,H_n$, for their free versions $S_n^+,O_n^+,B_n^+,H_n^+$, and for the half-liberated versions $O_n^*,H_n^*$.  We first collect all the known computations of such determinants, along with complete and simplified proofs, and with generalizations where needed. We conjecture that all these determinants decompose as $D=\prod_\pi\varphi(\pi)$, with product over all associated partitions.
\end{abstract}

\maketitle

\section*{Introduction}

We discuss in this paper the computation of certain advanced representation theory invariants, for the main examples of ``easy quantum groups''. These are the groups $S_n,O_n,B_n,H_n$, their free versions $S_n^+,O_n^+,B_n^+,H_n^+$, and the half-liberated versions $O_n^*,H_n^*$. Here $S_n,O_n$ are the permutation and orthogonal groups, $B_n$ is the bistochastic group consisting of orthogonal matrices whose rows and columns sum to 1, and $H_n = \mathbb Z_2 \wr S_n$ is the hyperoctahedral group.  For a global introduction to these groups and quantum groups, we refer to our previous papers \cite{ez1}, \cite{ez2}, \cite{ez3}.

According to a paper of Weingarten \cite{wei}, further processed and generalized by Collins \cite{col}, then Collins and \'Sniady \cite{csn}, a number of advanced representation theory invariants of the quantum group are encoded in a certain associated matrix $G_{kn}$, called Gram matrix. For instance the inverse of the Gram matrix $G_{kn}$ is the Weingarten matrix $W_{kn}$, whose knowledge allows the full computation of the Haar functional. See \cite{wei}, \cite{col}, \cite{csn}.

Among these invariants, the central object is the Gram matrix determinant, $\det(G_{kn})$. For instance the roots of $\det(G_{kn})$ are the poles of the Weingarten function $W_{kn}$, and the knowledge of these numbers clarifies the invertibility assumptions in \cite{ez1}, \cite{ez2}, \cite{ez3}.

The quantity $\det(G_{kn})$ appears in fact naturally in relation with many other questions, and its exact computation a well-known problem. A first purpose of the present work is to collect all the available formulae from the literature, and to write them down by using our unified ``easy quantum group'' formalism, along with complete, simplified proofs.

The basic example of such a formula is that for $S_n,H_n,H_n^*$. Here the Gram matrix is upper triangular, up to a simple determinant-preserving operation. The determinant, computed by Jackson \cite{jac} and Lindstr\"om \cite{lin}, is as follows:
$$\det(G_{kn})=\prod_{\pi\in \mathcal P(k)}\frac{n!}{(n-|\pi|)!}$$

Here $\mathcal P(k)$ is the set of partitions associated to the quantum group, namely all partitions for $S_n$, all partitions with even blocks for $H_n$, and all partitions with blocks having the same number of odd and even legs for $H_n^*$, and $|.|$ is the number of blocks.

In the general case the situation is much more complicated. However, one may still wonder for a general decomposition result, of the following type:
$$\det(G_{kn})=\prod_{\pi\in \mathcal P(k)}\varphi(\pi)$$

This question is of course quite vague, depending on how explicit we would like our functions $\varphi$ to be. For instance a natural requirement would be that in the case of a liberation $G_n\to G_n^+$, the functions $\varphi$ are related by an induction/restriction procedure.

This kind of specialized question appears to be quite difficult. In this paper we will obtain some preliminary decomposition results of the above type, as follows:
\begin{enumerate}
\item For $O_n,B_n,O_n^*$ a formula comes from the work of Collins-Matsumoto \cite{cma} and Zinn-Justin \cite{zin}. The natural decomposition here is over Young diagrams, and in principle one can pass to partitions by applying a certain surjective map.

\item For $O_n^+,B_n^+,S_n^+,H_n^+$ we use the work of  Di Francesco, Golinelli and Guitter \cite{df1}, \cite{df2}, \cite{dg1}, \cite{dg2}, Tutte \cite{tut} and Dahab \cite{dah}. We will obtain some evidence towards the existence of contributions $\varphi(\pi)$, of ``trigonometric'' nature.
\end{enumerate}

We will make as well a number of speculations in relation with quantum group/planar algebra methods, and with spectral measure/orthogonal polynomial interpretations. 

As a conclusion, there is still a lot of work to be done, mostly towards the conceptual understanding, at the level of Gram determinants, of the operation $G_n\to G_n^+$.

The paper is organized as follows: 1-2 are preliminary sections, in 3-4 we discuss the classical and half-liberated cases, and in 5-6 we discuss the free case. The final sections, 7-9, contain a number of speculations on the formulae, and a few concluding remarks.

\subsection*{Acknowledgements}

This work was started at the Bedlewo 2009 workshop ``Noncommutative harmonic analysis'', and we are highly grateful to M. Bo\.zejko for the invitation, and for several stimulating discussions. S.C. would like to thank the Paul Sabatier University and the Cergy-Pontoise University, where another part of this work was done. The work of T.B. was supported by the ANR grants ``Galoisint'' and ``Granma'', and the work of S.C. was partially supported by an NSF Postdoctoral Fellowship.

\section{Easy quantum groups}

Let $\mathcal P_s$ be the category of all partitions. That is, $\mathcal P_s(k,l)$ is the set of partitions between an upper row of $k$ points and a lower row of $l$ points, and the categorical operations are the horizontal and vertical concatenation, and the upside-down turning.

A category of partitions $\mathcal P \subset \mathcal P_s$ is by definition a collection of sets $\mathcal P(k,l)\subset \mathcal P_s(k,l)$, which is stable under the categorical operations. We have the following examples.

\begin{proposition}
The following are categories of partitions:
\begin{enumerate}
\item $\mathcal P_o/\mathcal P_o^+$: all pairings/all noncrossing pairings.

\item $\mathcal P_o^*$: pairings with each string having an odd leg and an even leg.

\item $\mathcal P_b/\mathcal P_b^+$: singletons plus pairings/noncrossing pairings.

\item $\mathcal P_s/\mathcal P_s^+$: all partitions/all noncrossing partitions.

\item $\mathcal P_h/\mathcal P_h^+$: partitions/noncrossing partitions with blocks of even size.

\item $\mathcal P_h^*$: partitions with blocks having the same number of odd and even legs.
\end{enumerate}
\end{proposition}

\begin{proof}
This is clear from definitions. Note that $\mathcal P_g^\times$ corresponds via Tannakian duality \cite{wo1}, \cite{wo2} to the easy quantum group $G^\times=(G_n^\times)$, with the notations in \cite{ez1}, \cite{bsp}.
\end{proof}

We use the notation $\mathcal P(k)=\mathcal P(0,k)$. We denote by $\vee$ and $\wedge$ the set-theoretic sup and inf of partitions, always taken with respect to $\mathcal P_s$, and by $|.|$ the number of blocks.

\begin{definition}
Associated to any category of partitions $\mathcal P$ and to any numbers $k,n\geq 0$ are the following matrices, with entries indexed by $\pi,\sigma\in \mathcal P(k)$:
\begin{enumerate}
\item Gram matrix: $G_{kn}(\pi,\sigma)=n^{|\pi\vee\sigma|}$.

\item Weingarten matrix: $W_{kn}=G_{kn}^{-1}$.
\end{enumerate}
\end{definition}

In order for $G_{kn}$ to be invertible, $n$ must be big enough, and $n\geq k$ is known to be sufficient. The precise bounds depend on the category of partitions, and can be deduced from the various explicit formulae of $\det(G_{kn})$, to be given later on in this paper.

The interest in the above matrices comes from the fact that in the case $\mathcal P= \mathcal P_g^\times$, they describe the integration over the corresponding easy quantum group $G_n^\times$.

\begin{theorem}
We have the Weingarten formula
$$\int_{G_n^\times}u_{i_1j_1}\ldots u_{i_kj_k}\,du=\sum_{\pi,\sigma\in \mathcal P_g^\times(k)}\delta_\pi(i)\delta_\sigma(j)W_{kn}(\pi,\sigma)$$
where the $\delta$ symbols are $0$ or $1$, depending on whether the indices fit or not.
\end{theorem}

\begin{proof}
This follows by using a classical argument from \cite{wei}, \cite{csn}. See \cite{ez1}, \cite{bsp}.
\end{proof}

The exact computation of the Weingarten matrix is a quite subtle problem. A precise result is available only in the finite group case, where the formula is given in terms of the M\"obius function $\mu$ on $\mathcal P$ as follows.

\begin{proposition}
For $S_n,H_n$ the Weingarten function is given by
$$W_{kn}(\pi,\sigma)=\sum_{\tau\leq\pi\wedge\sigma}\mu(\tau,\pi)\mu(\tau,\sigma)\frac{(n-|\tau|)!}{n!}$$
and satisfies $W_{kn}(\pi,\sigma)=n^{-|\pi\wedge\sigma|}(
\mu(\pi\wedge\sigma,\pi)\mu(\pi\wedge\sigma,\sigma)+O(n^{-1}))$.
\end{proposition}

\begin{proof}
The first assertion follows from the Weingarten formula: in that formula the integrals on the left are known, and this allows the computation of the right term, via the M\"obius inversion formula. The second assertion follows from the first one.
\end{proof}

In the general case we have the following result, which is useful for applications.

\begin{proposition}
For $\pi\leq\sigma$ we have the estimate
$$W_{kn}(\pi,\sigma)=n^{-|\pi|}(\mu(\pi,\sigma)+O(n^{-1}))$$
and for $\pi,\sigma$ arbitrary we have $W_{kn}(\pi,\sigma)=O(n^{|\pi\vee\sigma|-|\pi|-|\sigma|})$.
\end{proposition}

\begin{proof}
Once again this follows by using a classical argument, see \cite{ez1}.
\end{proof}

\section{Gram determinants}

In this paper, we will be mainly interested in the computation of $\det(G_{kn})$. Let us being with some simple observations, coming from definitions.

\begin{proposition}
Let $D_k(n)=\det(G_{kn})$, viewed as element of $\mathbb Z[n]$.
\begin{enumerate}
\item $D_k$ is monic, of degree $s_k=\sum_{\pi\in \mathcal P(k)}|\pi|$.

\item We have $n^{b_k}|D_k$, where $b_k=\# \mathcal P(k)$.
\end{enumerate}
\end{proposition}

\begin{proof}
(1) This follows from $|\pi\vee\sigma|\leq|\pi|$, with equality if and only if $\sigma\leq\pi$. Indeed, from the inequality we get $\deg(D_k)\leq s_k$. Now the coefficient of $n^{s_k}$ is the signed number of permutations $f:\mathcal P(k)\to \mathcal P(k)$ satisfying $f(\pi)\leq\pi$ for any $\pi$, and since there is only one such permutation, namely the identity, we obtain that this coefficient is 1.

(2) This is clear from the definition of $D_k$, and from $|\pi\vee\sigma|\geq 1$.
\end{proof}

The above result raises the question of computing the numbers $b_k=\#\mathcal P(k)$ and $s_k=\sum_{\pi\in \mathcal P(k)}|\pi|$. It is convenient here to introduce as well the related numbers $m_k=s_k/b_k$ and $a_k=2s_k-kb_k=(2m_k-k)b_k$, which will appear several times in what follows.

\begin{proposition}
The numbers $b_k,s_k,m_k,a_k$ are as follows:
\begin{enumerate}
\item $O_n,O_n^*,O_n^+$: $b_{2l}=(2l)!!,l!,\frac{1}{l+1}\binom{2l}{l}$, $s_{2l}=lb_{2l}$, $m_{2l}=l$, $a_{2l}=0$.

\item $S_n$: $b_k=$ Bell, $s_k=b_{k+1}-b_k$, $m_k=\frac{b_{k+1}}{b_k}-1$, $a_k=2b_{k+1}-(k+2)b_k$.

\item $S_n^+$: $b_k=\frac{1}{k+1}\binom{2k}{k}$, $s_k=\frac{1}{2}\binom{2k}{k}$, $m_k=\frac{k+1}{2}$, $a_k=b_k$.

\item $H_n^+$: $b_{2l}=\frac{1}{2l+1}\binom{3l}{l}$, $s_{2l}=\binom{3l-1}{l-1}$, $m_{2l}=\frac{2l+1}{3}$, $a_{2l}=-2\binom{3l-1}{l-2}$.
\end{enumerate}
\end{proposition}

\begin{proof}
All these results are well-known.
\end{proof}

For the remaining quantum groups, namely $B_n,B_n^+,H_n,H_n^*$, the numbers $b_k,s_k,m_k,a_k$ are given by quite complicated formulae. The best approach to their computation is via the trace of the Gram matrix, and its analytic interpretations.

So, let us first reformulate Proposition 2.1, in the following way.

\begin{proposition}
With $D_k(n)=\det(G_{kn})$ and $T_k(t)=Tr(G_{kt})$, we have:
\begin{enumerate}
\item $D_k(n)=n^{s_k}(1+O(n^{-1}))$ as $n \to \infty$, where $s_k=T_k'(1)$.

\item $D_k(n)=O(n^{b_k})$ as $n \to 0$, where $b_k=T_k(1)$.
\end{enumerate}
\end{proposition}

\begin{proof}
This is indeed just a reformulation of Proposition 2.1, using a variable $t$ around 1. Note that in (2) we regard the variable $n$ as a formal parameter, going to $0$.
\end{proof}

The trace can be understood in terms of the associated Stirling numbers. 

\begin{proposition}
We have the formula
$$T_k(t)=\sum_{r=1}^kS_{kr}t^r$$
where $S_{kr}=\#\{\pi\in \mathcal P(k):|\pi|=r\}$ are the Stirling numbers.
\end{proposition}

\begin{proof}
This is clear from definitions.
\end{proof}

Another interpretation of the trace, analytic this time, is as follows.

\begin{proposition}
For any $t\in(0,1]$ we have the formula
$$T_k(t)=\lim_{n\to\infty}\int_{G_n^\times}\chi_t^k$$
where $\chi_t=\sum_{i=1}^{[tn]}u_{ii}$ are the truncated characters of the quantum group.
\end{proposition}

\begin{proof}
As explained in \cite{bsp}, \cite{ez1}, this follows from the Weingarten formula.
\end{proof}

In general, the Stirling numbers $S_{kr}$ and the trace $T_k(t)$ are given by quite complicated formulae, unless we are in the situation of one of the quantum groups in Proposition 2.2. Here these invariants are well-known in the $O,S$ cases, and for $H^+$ we have:
$$T_{2l}(t)=\sum_{r=1}^l\frac{1}{r}\binom{l-1}{r-1}\binom{2l}{r-1}t^r$$

See \cite{bb+}. In general now, the conceptual result concerns the asymptotic measures of truncated characters, i.e. the probability measures $\mu_t$ satisfying $T_k(t)=\int x^kd\mu_t(x)$. 

\begin{theorem}
The asymptotic measures of truncated characters are as follows:
\begin{enumerate}
\item $S_n/S_n^+$: Poisson/free Poisson.

\item $O_n/O_n^+$: Gaussian/semicircular.

\item $H_n/H_n^+$: Bessel/free Bessel.

\item $B_n/B_n^+$: shifted Gaussian/shifted semicircular.

\item $O_n^*/H_n^*$: symmetrized Rayleigh/squeezed $\infty$-Bessel.
\end{enumerate}
\end{theorem}

\begin{proof}
The one-parameter measures in the statement are best found via a direct computation, by using classical and free cumulants. See \cite{bsp}, \cite{ez1}, \cite{ez2}.
\end{proof}

\section{The basic formula}

We discuss now the explicit computation of the Gram determinants. The basic formula here, coming from the work of Jackson \cite{jac} and Lindstr\"om \cite{lin}, is as follows. 

\begin{theorem}
For $S_n,H_n,H_n^*$ we have
$$\det(G_{kn})=\prod_{\pi\in \mathcal P(k)}\frac{n!}{(n-|\pi|)!}$$
where $|.|$ is the number of blocks.
\end{theorem}

\begin{proof}
We use the fact that the partitions have the property of forming semilattices under $\vee$. The proof uses the upper triangularization procedure in \cite{lin} together with the explicit knowledge of the M\"obius function on $\mathcal P(k)$ as in \cite{jac}. Consider the following matrix, obtained by making determinant-preserving operations:
$$G_{kn}'(\pi,\sigma)=\sum_{\pi\leq\tau}\mu(\pi,\tau)n^{|\tau\vee\sigma|}$$

It follows from the M\"obius inversion formula that we have:
$$G_{kn}'(\pi,\sigma)=
\begin{cases}
n(n-1)\ldots(n-|\sigma|+1)&{\rm if}\ \pi\leq\sigma\\
0&{\rm if\ not}
\end{cases}$$

Thus the matrix is upper triangular, and by computing the product on the diagonal we obtain the formula in the statement.
\end{proof}

A first remarkable feature of the above result is that the determinant for $S_n,H_n,H_n^*$ can be computed from the trace: indeed, the Gram trace gives the Stirling numbers, which in turn give the Gram determinant. However, the connecting formula is quite complicated, so let us just record here an improvement of the first estimate in Proposition 2.3.

\begin{proposition}
With $D_k(n)=\det(G_{kn})$ and $T_k(t)=Tr(G_{kt})$ we have
$$D_k(n)=n^{s_k}\left(1-\frac{z_k}{2}\,n^{-1}+O(n^{-2})\right)$$
where $s_k=T_k'(1)$ and $z_k=T_k''(1)$.
\end{proposition}

\begin{proof}
In terms of Stirling numbers, the formula in Theorem 3.1 reads:
$$D_k(n)=\prod_{r=1}^k\left(\frac{n!}{(n-r)!}\right)^{S_{kr}}$$

We use now the following basic estimate:
$$\frac{n!}{(n-r)!}=n^r\prod_{s=1}^{r-1}\left(1-\frac{s}{n}\right)=n^r\left(1-\frac{r(r-1)}{2}\,n^{-1}+O(n^{-2})\right)$$

Together with $T_k(t)=\sum_{r=1}^kS_{kr}t^r$, this gives the result. 
\end{proof}

The above discussion raises the general question on whether the Gram determinant can be computed or not from the Gram trace, or from the measures in Theorem 2.6. 

Since the connecting formula for $S_n,H_n,H_n^*$ is already quite complicated, let us formulate for the moment a more modest conjecture, as follows.

\begin{conjecture}
For any easy quantum group we have a formula of type
$$\det(G_{kn})=\prod_{\pi\in \mathcal P(k)}\varphi(\pi)$$
with the ``contributions'' being given by an explicit function $\varphi:\mathcal P(k)\to\mathbb Q(n)$.
\end{conjecture}

This statement is of course quite vague, depending of the meaning of the above word ``explicit''. As already mentioned, one would expect $\varphi$ to come from the Gram trace, or from the Stirling numbers, or, even better, from the measures in Theorem 2.6.  Such a decomposition could potentially clarify the behavior of the Gram determinants under the ``liberation'' procedure $G \to G^+$.

This kind of general question appears to be quite difficult. In what follows we will obtain some evidence towards such general decomposition results.

\section{The orthogonal case}

We discuss now the cases $O,B,O^*$. Here the combinatorics is that of the Young diagrams. We denote by $|.|$ the number of boxes, and we use quantity $f^\lambda$, which gives the number of standard Young tableaux of shape $\lambda$.

\begin{theorem}
For $O_n$ we have
$$\det(G_{kn})=\prod_{|\lambda|=k/2}f_n(\lambda)^{f^{2\lambda}}$$
where $f_n(\lambda)=\prod_{(i,j)\in\lambda}(n+2j-i-1)$. 
\end{theorem}

\begin{proof}
This follows from the results of Collins and Matsumoto \cite{cma} and Zinn-Justin \cite{zin}. Indeed, it is known from \cite{zin} that the Gram matrix is diagonalizable, as follows:
$$G_{kn}=\sum_{|\lambda|=k/2}f_n(\lambda)P_{2\lambda}$$

Here $1=\Sigma P_{2\lambda}$ is the standard partition of unity associated to the Young diagrams having $k/2$ boxes, and the coefficients $f_n(\lambda)$ are those in the statement. Now since we have $Tr(P_{2\lambda})=f^{2\lambda}$, this gives the result.
\end{proof}

\begin{theorem}
For $B_n$ we have
$$\det(G_{kn})=n^{a_k}\prod_{|\lambda|\leq k/2}f_n(\lambda)^{\binom{k}{2|\lambda|}f^{2\lambda}}$$
where $a_k=\sum_{\pi\in \mathcal P(k)}(2|\pi|-k)$, and $f_n(\lambda)=\prod_{(i,j)\in\lambda}(n+2j-i-2)$.
\end{theorem}

\begin{proof}
We recall from \cite{bsp} that we have an isomorphism $B_n\simeq O_{n-1}$, given by $u=v+1$, where $u,v$ are the fundamental representations of $B_n,O_{n-1}$. We get:
$$Fix(u^{\otimes k})
=Fix\left((v+1)^{\otimes k}\right)
=Fix\left(\sum_{r=0}^k\binom{k}{r}v^{\otimes r}\right)$$

Now if we denote by ${\rm det}',f'$ the objects in Theorem 4.1, we obtain:
$$\det(G_{kn})
=n^{a_k}\prod_{r=1}^k{\rm det}'(G_{r,n-1})^{\binom{k}{r}}
=n^{a_k}\prod_{r=1}^k\left(\prod_{|\lambda|=r/2}f'_{n-1}(\lambda)^{f^{2\lambda}}\right)^{\binom{k}{r}}$$

This gives the formula in the statement.
\end{proof}

\begin{theorem}
For $O_n^*$ we have
$$\det(G_{kn})=\prod_{|\lambda|=k/2}f_n(\lambda)^{{f^\lambda}^2}$$
where $f_n(\lambda)=\prod_{(i,j)\in\lambda}(n+j-i)$.
\end{theorem}

\begin{proof}
We use the isomorphism of projective versions $PO_n^*=PU_n$, established in \cite{bve}. This isomorphism shows that the Gram matrices for $O_n^*$ are the same as those for $U_n$. But for $U_n$ it is known from \cite{zin} that the Gram matrix is diagonalizable, as follows:
$$G_{kn}=\sum_{|\lambda|=k/2}f_n(\lambda)P_\lambda$$

Here $1=\Sigma P_\lambda$ is the standard partition of unity associated to the Young diagrams having $k/2$ boxes, and the coefficients $f_n(\lambda)$ are those in the statement. Now since we have $Tr(P_\lambda)={f^\lambda}^2$, this gives the result.
\end{proof}

Observe that the above results provide a kind of answer to Conjecture 3.3, but with the Young diagrams contributing to the determinant, instead of the partitions. The remaining problems are to find the relevant surjective map from diagrams to partitions, and to see if the above formulae further simplify by using this surjective map.

\section{Meander determinants}

In this section we discuss the computation of the Gram matrix determinant, in the free cases $O_n^+,B_n^+,S_n^+,H_n^+$. Let $P_r$ be the Chebycheff polynomials, given by $P_0=1,P_1=n$ and $P_{r+1}=nP_r-P_{r-1}$. Consider also the following numbers, depending on $k,r\in\mathbb Z$:
$$f_{kr}=\binom{2k}{k-r}-\binom{2k}{k-r-1}$$

We set $f_{kr}=0$ for $k\notin\mathbb Z$. The following key result was proved in \cite{dg1}.

\begin{theorem}
For $O_n^+$ we have
$$\det(G_{kn})=\prod_{r=1}^{[k/2]}P_r(n)^{d_{k/2,r}}$$
where $d_{kr}=f_{kr}-f_{k+1,r}$.
\end{theorem}

\begin{proof}
As already mentioned, the result is from \cite{dg1}. We present below a short proof. The result holds when $k$ is odd, all the exponents being 0, so we assume that $k$ is even.

{\bf Step 1.} We use a general formula of type $G_{kn}(\pi,\sigma)=<f_\pi,f_\sigma>$. 

Let $\Gamma$ be a locally finite bipartite graph, with distinguished vertex 0 and adjacency matrix $A$, and let $\mu$ be an eigenvector of $A$, with eigenvalue $n$. Let $L_k$ be the set of length $k$ loops $l=l_1\ldots l_k$ based at $0$, and $H_k=span(L_k)$. For $\pi\in \mathcal P_{o^+}(k)$ define $f_\pi\in H_k$ by:
$$f_\pi=\sum_{l\in L_k}\left(\prod_{i\sim_\pi j}\delta(l_i,l_j^o)\gamma(l_i)\right)l$$

Here $e\to e^o$ is the edge reversing, and the ``spin factor'' is $\gamma=\sqrt{\mu(t)/\mu(s)}$, where $s,t$ are the source and target of the edges. The point is that we have $G_{kn}(\pi,\sigma)=<f_\pi,f_\sigma>$. We refer to \cite{mar}, \cite{jo1}, \cite{gjs} for full details regarding this formula.

{\bf Step 2.} With a suitable choice of $(\Gamma,\mu)$, we obtain a fomula of type $G_{kn}=T_{kn}T_{kn}^t$.

Indeed, let us choose $\Gamma=\mathbb N$ to be the Cayley graph of $O_n^+$, and the eigenvector entries $\mu(r)$ to be the Chebycheff polynomials $P_r(n)$, i.e. the orthogonal polynomials for $O_n^+$.

In this case, we have a bijection $\mathcal P_{o^+}(k)\to L_k$, constructed as follows.  For $\pi \in \mathcal P_{o^+}(k)$ and $0\leq i\leq k$ we define $h_\pi(i)$ to be the number of $1\leq j \leq i$ which are joined by $\pi$ to a number strictly larger than $i$.  We then define a loop $l(\pi)=l(\pi)_1\ldots l(\pi)_k$, where $l(\pi)_i$ is the edge from $h_{\pi}(i-1)$ to $h_{\pi}(i)$. Consider now the following matrix:
$$T_{kn}(\pi,\sigma)=\prod_{i\sim_\pi j}\delta(l(\sigma)_i,l(\sigma)_j^o)\gamma(l(\sigma)_i)$$

We have $f_\pi=\sum_\sigma T_{kn}(\pi,\sigma)\cdot l(\sigma)$, so we obtain as desired $G_{kn}=T_{kn}T_{kn}^t$.

{\bf Step 3.} We show that, with suitable conventions, $T_{kn}$ is lower triangular.

Indeed, consider the partial order on $\mathcal P_{o^+}(k)$ given by $\pi\leq\sigma$ if $h_\pi(i)\leq h_\sigma(i)$ for $i=1,\ldots,k$.  Our claim is that $\sigma\not\leq\pi$ implies $T_{kn}(\pi,\sigma)=0$.

Indeed, suppose that $\sigma\not\leq\pi$, and let $j$ be the least number with $h_\sigma(j)>h_\pi(j)$. Note that we must have $h_\sigma(j-1)=h_\pi(j-1)$ and $h_\sigma(j)=h_\pi(j)+2$. It follows that we have $i\sim_\pi j$ for some $i<j$.  From the definitions of $T_{kn}$ and $l(\sigma)$, if $T_{kn}(\pi,\sigma)\neq 0$ then we must have $h_\sigma(i-1)=h_\sigma(j)=h_\pi(j)+2$. But we also have $h_\pi(i-1)=h_\pi(j)$, so that $h_\sigma(i-1)=h_\pi(i-1)+2$, which contradicts the minimality of $j$.

{\bf Step 4.} End of the proof, by computing the determinant of $T_{kn}$.

Since $T_{kn}$ is lower triangular we have:
$$\det(T_{kn}) 
=\prod_\pi T_{kn}(\pi,\pi)
=\prod_\pi\prod_{i\sim_\pi j}
\sqrt{\frac{P_{h_{\pi(i)}}}{P_{h_{\pi(i)-1}}}}
=\prod_{r=1}^{k/2}P_r^{e_{kr}/2}$$

Here the exponents appearing on the right are by definition as follows:
$$e_{kr}=\sum_\pi\sum_{i\sim_\pi j}\delta_{h_\pi(i),r}-\delta_{h_\pi(i),r+1}$$

Our claim now, which finishes the proof, is that for $1\leq r\leq k/2$ we have:
$$\sum_\pi\sum_{i\sim_\pi j}\delta_{h_\pi(i)r}=f_{k/2,r}$$

Indeed, note that the left term counts the number of times that the edge $(r,r+1)$ appears in all loops in $L_k$. Define a shift operator $S$ on the edges of $\Gamma$ by $S(s,t)=(s+1,t+1)$.  Given a loop $l=l_1\ldots l_k$ and $1\leq s\leq k$ with $l_s=(r,r+1)$, define a path $S^r(l_s)\ldots S^r(l_k)l_{s-1}^o\ldots l_1^o$. Observe that this is a path in $\Gamma$ from $2r$ to $0$ whose first edge is $(2r,2r+1)$ and first reaches $r-1$ after $k-s+1$ steps.

Conversely, given a path $f_1\ldots f_k$ in $\Gamma$ from $2r$ to $0$ whose first edge is $(2r,2r+1)$ and first reaches $r-1$ after $s$ steps, define a loop $f_k^o\ldots f_s^oS^{-r}(f_1)\ldots S^{-r}(f_{s-1})$.  Observe that this is a loop in $\Gamma$ based at 0 whose $k-s+1$ edge is $(r,r+1)$.

These two operations are inverse to each other, so we have established a bijection between $k$-loops in $\Gamma$ based at $0$ whose $s$-th edge is $(r,r+1)$ and $k$-paths in $\Gamma$ from $2r$ to $0$ whose first edge is $(2r,2r+1)$ and which first reach $r-1$ after $k-s+1$ steps.

It follows that the left hand side is equal to the number of paths in $\Gamma=\mathbb N$ from $2r$ to $0$ whose first edge is $(2r, 2r+1)$.  By the usual reflection trick, this is the difference of binomials defining $f_{k/2,r}$, and we are done.
\end{proof}

We use the notation $a_k=\sum_{\pi\in \mathcal P(k)}(2|\pi|-k)$, which already appeared in section 2.

\begin{theorem}
For $B_n^+$ we have:
$$\det(G_{kn})=n^{a_k}\prod_{r=1}^{[k/2]}P_r(n-1)^{\sum_{l=1}^{[k/2]}\binom{k}{2l}d_{lr}}$$
\end{theorem}

\begin{proof}
We have $B_n^+\simeq O_{n-1}^+$, see e.g. \cite{rau}, so we can use the same method as in the proof of Theorem 4.2. By using prime exponents for the various $O_n^+$-related objects, we get:
$$\det(G_{kn})
=n^{a_k}\prod_{l=1}^{[k/2]}{\rm det}'(G_{2l,n-1})^{\binom{k}{2l}}
=n^{a_k}\prod_{l=1}^{[k/2]}\left(\prod_{r=1}^lP_r(n-1)^{d_{lr}}\right)^{\binom{k}{2l}}$$

Together with Theorem 5.1, this gives the formula in the statement.
\end{proof}

\begin{theorem}
For $S_n^+$ we have: 
$$\det(G_{kn})=(\sqrt{n})^{a_k}\prod_{r=1}^kP_r(\sqrt{n})^{d_{kr}}$$
\end{theorem}

\begin{proof}
Let $\pi\to\widetilde\pi$ be the ``cabling'' operation, obtained by collapsing neighbors. According to the results of Kodiyalam-Sunder \cite{ksu} and Chen-Przytycki \cite{cpr}, we have:
$$|\pi\vee\sigma|=k/2+2|\widetilde\pi\vee\widetilde\sigma|-|\widetilde\pi|-|\widetilde\sigma|$$

In terms of Gram matrices we get $G_{kn}=D_{kn}G'_{2k,\sqrt{n}}D_{kn}$, where $D_{kn}=diag(n^{|\widetilde\pi|/2-k/4})$, and where $G'$ is the Gram matrix for $O_n^+$, so the result follows from Theorem 5.1.
\end{proof}

\begin{theorem}
For $H_n^+$ we have the formula
$$\det(G_{kn})=(\sqrt{n})^{a_k}\prod_{r=1}^{[k/2]}P_r(\sqrt{n})^{2d_{k/2,r}'}$$
with $d_{sr}'=f_{sr}'-f_{s,r+1}'$, where 
$f_{sr}'=\binom{3s}{s-r}-\binom{3s}{s-r-1}$ for $s\in\mathbb Z$, $f_{sr}'=0$ for $s\notin\mathbb Z$.
\end{theorem}

\begin{proof}
According to \cite{bbc}, the diagrams for $H_n^+$ are the ``cablings'' of the Fuss-Catalan diagrams \cite{bjo}, so we can use the same method as in the previous proof. So, by using the above formula from \cite{ksu}, \cite{cpr}, we have $G_{kn}=D_{kn}G'_{2k,\sqrt{n}}D_{kn}$, where $D_{kn}=diag(n^{|\widetilde\pi|/2-k/4})$, and where $G'$ is the Gram determinant for the Fuss-Catalan algebra. But this latter determinant was computed by Di Francesco in \cite{df2}, and this gives the result.
\end{proof}

\section{Algebraic manipulations}

In this section we perform some algebraic manipulations on the formulae found in the previous sections. Consider the quantity $a_k=\sum_{\pi\in \mathcal P(k)}(2|\pi|-k)$, which already appeared, several times. Then $n^{a_k}$ is a true ``contribution'', in the sense of Conjecture 3.3. 

We will prove here that a $n^{a_k}$ factor appears naturally, in all the 10 formulae of Gram determinants. This can be regarded as a piece of evidence towards Conjecture 3.3.

In the classical and half-liberated cases there is no need for supplementary work in order to isolate this $n^{a_k}$ factor, and the unified result is as follows.

\begin{theorem}
In the classical and half-liberated cases, we have
\begin{eqnarray*}
S_n,H_n,H_n^*:\quad\det(G_{kn})&=&n^{a_k}\prod_{\pi\in \mathcal P(k)}\frac{n^{k-2|\pi|}n!}{(n-|\pi|)!}\\
O_n:\quad\det(G_{kn})&=&n^{a_k}\prod_{|\lambda|=k/2}f_n(\lambda)^{f^{2\lambda}}\\
B_n:\quad\det(G_{kn})&=&n^{a_k}\prod_{|\lambda|\leq k/2}f_n'(\lambda)^{\binom{k}{2|\lambda|}f^{2\lambda}}\\
O_n^*:\quad\det(G_{kn})&=&n^{a_k}\prod_{|\lambda|=k/2}f_n''(\lambda)^{{f^\lambda}^2}
\end{eqnarray*}
where $f_n^\circ(\lambda)=\prod_{(i,j)\in\lambda}(n+j-i+\varphi^\circ)$, with $\varphi=j-1,\varphi'=j-2,\varphi''=0$.
\end{theorem}

\begin{proof}
This is a reformulation of the results in section 4, by using $a_k=0$ for $O_n,O_n^*$. 
\end{proof}

In order to process the formulae in section 5, we need the following technical result. 

\begin{proposition}
The Chebycheff polynomials $P_r$ have the following properties:
\begin{enumerate}
\item $P_r(n-1)=Q_r(n)$, with $Q_0=1,Q_1=n-1$ and $Q_{r+1}=(n-1)Q_r-Q_{r-1}$.

\item $P_{2l}(n)=R_{2l}(n^2)$, with $R_0=1,R_2=n-1$ and $R_{2l+2}=(n-2)R_{2l}-R_{2l-2}$.

\item $P_{2l+1}(n)=nR_{2l+1}(n^2)$, with $R_1=1,R_3=n-2$ and $R_{2l+3}=(n-2)R_{2l+1}-R_{2l-1}$.

\item $P_{2l}(n)=n^{-l}S_{2l}(n^2)^{1/2}$, with $S_0=1,S_2=n(n-1)^2$ and so on.

\item $P_{2l+1}(n)=n^{-l}S_{2l+1}(n^2)^{1/2}$, with $S_1=n,S_3=n^2(n-2)^2$ and so on. 
\end{enumerate}
\end{proposition}

\begin{proof}
This is routine. As pointed out in section 7 below, $Q_r$ are the orthogonal polynomials for $B_n^+$, and $R_{2l}$ are the orthogonal polynomials for $S_n^+$. As for the polynomials $S_r$, these are some technical objects, introduced in relation with the $H_n^+$ computation.
\end{proof}

\begin{theorem}
In the free cases, we have
\begin{eqnarray*}
O_n^+:\quad\det(G_{kn})&=&n^{a_k}\prod_{r=1}^{[k/2]}P_r(n)^{d_{k/2,r}^1}\\
B_n^+:\quad\det(G_{kn})&=&n^{a_k}\prod_{r=1}^{[k/2]}Q_r(n)^{\sum_{l=1}^{[k/2]}\binom{k}{2l}d_{lr}^1}\\
S_n^+:\quad\det(G_{kn})&=&n^{a_k}\prod_{r=1}^kR_r(n)^{d_{kr}^1}\\
H_n^+:\quad\det(G_{kn})&=&n^{a_k}\prod_{r=1}^{[k/2]}S_r(n)^{d_{k/2,r}^2}
\end{eqnarray*}
where $d_{kr}^i=f_{kr}^i-f_{k,r+1}^i$, with $f_{kr}^i=\binom{(i+1)k}{k-r}-\binom{(i+1)k}{k-r-1}$ for $k\in\mathbb Z$, and $f_{kr}^i=0$ for $k\notin\mathbb Z$.
\end{theorem}

\begin{proof}
The $O_n^+$ formula is the one in Theorem 5.1, with a $n^{a_k}=1$ factor inserted.

The $B_n^+$ formula is the one in Theorem 5.2, with $P_r(n-1)$ replaced by $Q_r(n)$.

For the $S_n^+$ formula, we use Theorem 5.3. By replacing the Chebycheff polynomials $P_{2l},P_{2l+1}$ by the polynomials $R_{2l},R_{2l+1}$ from Proposition 6.2, we get:
$$\det(G_{kn})
=(\sqrt{n})^{a_k}\prod_{r=1}^kP_r(\sqrt{n})^{d_{kr}^1}
=(\sqrt{n})^{a_k}\sqrt{n}^{\sum_{l=1}^{[(k+1)/2]}d_{k,2l-1}^1}\prod_{r=1}^kR_r(n)^{d_{kr}}$$

Now recall from Proposition 2.2 that $a_k=\frac{1}{k+1}\binom{2k}{k}$. On the other hand a direct computation gives $\sum_{l=1}^{[(k+1)/2]}d_{k,2l-1}^1=\frac{1}{k+1}\binom{2k}{k}$, so we get the formula in the statement.

For the $H_n^+$ formula we use a similar method. With $k=2l$, Theorem 5.4 gives:
$$\det(G_{2l,n})=(\sqrt{n})^{a_{2l}}\prod_{r=1}^{l}P_r(\sqrt{n})^{2d_{lr}^2}
=(\sqrt{n})^{a_{2l}}(\sqrt{n})^{-2\sum_{s=2}^l[s/2]d_{ls}^2}\prod_{r=1}^{l}S_r(n)^{d_{lr}^2}$$

Now recall from Proposition 2.2 that $a_{2l}=-2\binom{3l-1}{l-2}$. On the other hand a direct computation gives $\sum_{s=2}^l[s/2]d_{ls}^2=\binom{3l-1}{l-2}$, so we get the formula in the statement.
\end{proof}

As a conclusion, the formulae in Theorem 6.1 and Theorem 6.3 are an intermediate step towards a general decomposition result of type $\det(G_{kn})=\prod_{\pi\in \mathcal P(k)}\varphi(\pi)$. We will come back to the question of finding such a general decomposition result in section 8 below.

\section{Orthogonal polynomials}

We present here a speculation in the free case, in relation with orthogonal polynomials. As we will see, this speculation works for $S_n^+,O_n^+,B_n^+$, but doesn't work for $H_n^+$.

\begin{definition}
The orthogonal polynomials for a real probability measure $\mu$ are the polynomials $Q_0,Q_1,Q_2,\ldots$ satisfying the following conditions:
\begin{enumerate}
\item $Q_k(n)=n^k+a_1n^{k-1}+\ldots+a_{k-1}n+a_k$, with $a_i\in\mathbb R$.

\item For any $k\neq l$ we have $\int Q_k(n)Q_l(n)\,d\mu(n)=0$.
\end{enumerate}
\end{definition}

The orthogonal polynomials can be constructed by using a recursive formula, of type $Q_{k+1}=(n-\alpha_k)Q_k-\beta_kQ_{k-1}$. Here the parameters $\alpha_k,\beta_k\in\mathbb R$ are uniquely determined by the linear equations coming from the fact that $Q_{k+1}$ must be orthogonal to $n^{k-1},n^k$. 

More precisely, by solving these two equations we obtain the following formulae, where the integral sign denotes the integration with respect to $\mu$:
$$\alpha_k=\frac{\int n^{k+1}Q_k}{\int n^kQ_k}-\frac{\int n^kQ_{k-1}}{\int n^{k-1}Q_{k-1}},\quad\beta_k=\frac{\int n^kQ_k}{\int n^{k-1}Q_{k-1}}$$

The numbers $\alpha_k,\beta_k$ are called Jacobi parameters of the sequence $\{Q_k\}$. Since $Q_0=1$, in order to describe $\{Q_k\}$ we just need to specify $Q_1$, and the Jacobi parameters.

The orthogonal polynomials for an easy quantum group are by definition those for the asymptotic measure of the main character, given in Theorem 2.6.

\begin{proposition}
The basic orthogonal polynomials are as follows:
\begin{enumerate}
\item $O_n$: here $Q_1=n$ and $Q_{k+1}=nQ_k-kQ_{k-1}$.

\item $B_n$: here $Q_1=n-1$ and $Q_{k+1}=(n-1)P_k-kQ_{k-1}$.

\item $O_n^*$: here $Q_1=n$ and $Q_{k+1}=nQ_k-[(k+1)/2]Q_{k-1}$.

\item $S_n$: here $Q_1=n-1$ and $Q_{k+1}=(n-k-1)Q_k-kQ_{k-1}$.

\item $O_n^+$: here $Q_1=n$ and $Q_{k+1}=nQ_k-Q_{k-1}$.

\item $B_n^+$: here $Q_1=n-1$ and $Q_{k+1}=(n-1)Q_k-Q_{k-1}$.

\item $S_n^+$: here $Q_1=n-1$ and $Q_{k+1}=(n-2)Q_k-Q_{k-1}$.
\end{enumerate}
\end{proposition}

\begin{proof}
This result is well-known, and easy to deduce from definitions. Note that all the polynomials in the above statement are versions of the polynomials appearing in (1,3,4,5), which are respectively the Hermite, Charlier  and Chebycheff polynomials.
\end{proof}

Let us go back now to the considerations in section 6. The polynomials $R_{2l}$ appearing in Proposition 6.2 are the orthogonal polynomials for $S_n^+$, and it is natural to call $\{R_n|n\in\mathbb N\}$ the family of ``extended orthogonal polynomials'' for $S_n^+$.

\begin{theorem}
In the $O_n^+,B_n^+,S_n^+$ cases we have a formula of type
$$\det(G_{kn})=n^{a_k}\prod_{r=1}^kQ_r(n)^{d_{kr}}$$
with $d_{kr}\in\mathbb N$, where $Q_r(n)$ are the corresponding extended orthogonal polynomials.
\end{theorem}

\begin{proof}
This follows from Theorem 6.3 and Proposition 7.2.
\end{proof}

Regarding now $H_n^+$, the combinatorics here is that of the Fuss-Catalan algebra \cite{bjo}, see \cite{bb+}, \cite{bbc}. Since $\mu$ is symmetric, the orthogonal polynomials are given by $Q_1=n$ and $Q_{k+1}=nQ_k-\beta_kQ_{k-1}$, where $\beta_k=\gamma_k/\gamma_{k-1}$, with $\gamma_k=\int n^kP_k$. The data is as follows:
\begin{center}
\begin{tabular}[t]{|l|l|l|l|l|l|l|l|l|l|l|l|l|l|l|l}
\hline &1&2&3&4&5&6&7&8\\
\hline $c_k$&1&3&12&55&273&1428&7752&43263\\
\hline $\gamma_k$&1&2&3&11/2&26/3&170/11&17.19/13&19.23/10\\
\hline $\beta_k$&1&2&3/2&11/6&52/33&15.17/11.13&11.19/130&13.23/170\\
\hline
\end{tabular}
\end{center}

\medskip

This suggests the following general formula:
$$\beta_k=
\begin{cases}
\displaystyle{\frac{3(3k-1)(3k+2)}{4(2k-1)(2k+1)}}&(k\ {\rm even})\\
\\
\displaystyle{\frac{3(3k-2)(3k+1)}{4(2k-1)(2k+1)}}&(k\ {\rm odd})
\end{cases}$$

The problem can be probably investigated by using techniques from \cite{hml}, \cite{leh}, \cite{mlo}. Our main problem is of course: what is the analogue of Theorem 7.3 for $H_n^+$?

Let us also mention that the computation of the orthogonal polynomials for $H_n,H_n^*$ looks like a quite difficult problem. Probably the good framework here is that of the quantum groups $H_n^{(s)}$ from \cite{ez1}, because at $s=2,\infty$ we have $H_n,H_n^*$.

We have as well the following question: is there a quantum group/planar algebra proof of Theorem 7.3, in the cases $B_n^+,S_n^+$? For $O_n^+$ this was done in Theorem 5.1.

\section{More manipulations}

We have seen in the previous section that the quantum group $H_n^+$ is somehow of a more complicated nature than the other quantum groups under consideration. 

In this section we restrict attention to $O_n^+,B_n^+,S_n^+$, and we further rearrange the formulae in Theorem 6.3. The idea comes from the formula of $O_n^+$. Indeed, the numbers $f_{kr}$ for $O_n^+$ count the $\mathcal P_{o^+}$ diagrams with $2r$ upper points and $2k$ lower points, with the property that each upper point is paired with a lower point. This kind of diagrams, called ``epi" in the paper of Jones, Shlyakhtenko and Walker \cite{jsw}, have the following generalization.

\begin{definition}
Let $\mathcal P$ be a category of partitions, and let $0\leq r\leq k$.
\begin{enumerate}
\item We let $\mathcal P^r(k)$ be the set of partitions $\sigma\in \mathcal P(r,k)$, with $0\leq r\leq k$, such that each upper point is connected to lower points only, and to at least one of them.

\item The elements of $\mathcal P^r(k)$ are called ``epi''. We let $\mathcal P^+(k)=\cup_{r=0}^k\mathcal P^r(k)$. For an epi $\sigma\in \mathcal P^r(k)$, we denote by $r(\sigma)=r$ the number of its upper legs.
\end{enumerate}
\end{definition}

With these notations, we can now state and prove our main result. This is a global formula for the Gram determinants associated to the quantum groups $O_n^+,B_n^+,S_n^+$.

\begin{theorem}
For $O_n^+,B_n^+,S_n^+$ we have the formula
$$\det(G_{kn})=n^{a_k}\prod_{\sigma\in \mathcal P^+(k)}\frac{F_{r(\sigma)}}{F_{r(\sigma)-1}}$$
where $F_r=P_{r/2},Q_{r/2},R_r$ are the corresponding extended orthogonal polynomials.
\end{theorem}

\begin{proof}
Observe first that the $F_r$ quantities in the statement make indeed sense. This is because the epi for $O_n^+,B_n^+$ must have an even number of upper legs.

(1) For $O_n^+$ we have $f_{sr}^1=\#\mathcal P^{2r}(2s)$, so the formula in Theorem 6.3 becomes:
$$\det(G_{kn})=n^{a_k}\prod_{r=0}^{[k/2]}P_r(n)^{f_{k/2,r}-f_{k/2,r+1}}
=n^{a_k}\prod_{r=0}^{[k/2]}P_r(n)^{\#\mathcal P^{2r}(k)-\#\mathcal P^{2r+2}(k)}$$

Now since we have $\mathcal P^+(k)=\cup_{r=0}^{[k/2]}\mathcal P^{2r}(k)$, we obtain the formula in the statement:
$$\det(G_{kn})=n^{a_k}\prod_{r=0}^{[k/2]}\left(\prod_{\sigma\in \mathcal P^{2r}(k)}\frac{Q_r(n)}{P_{r-1}(n)}\right)=n^{a_k}\prod_{\sigma\in \mathcal P^+(k)}\frac{P_{r(\sigma)/2}(n)}{P_{r(\sigma)/2-1}(n)}$$

(2) For $B_n^+$ the epi have, according to our definitions, singletons only in the lower row. Thus these epi can be counted as function of those for $O_n^+$, and we get:
$$\det(G_{kn})=n^{a_k}\prod_{r=0}^{[k/2]}Q_r(n)^{\sum_{l=1}^{[k/2]}\binom{k}{2l}d_{lr}^1}
=n^{a_k}\prod_{r=0}^{[k/2]}Q_r(n)^{\#\mathcal P^{2r}(k)-\#\mathcal P^{2r+2}(k)}$$

A similar manipulation as in (1) gives now the formula in the statement.

(3) For $S_n^+$ the epi are in standard bijection (via fatenning/collapsing of neighbors) with the epi for $O_n^+$. Thus the formula in Theorem 6.3 becomes:
$$\det(G_{kn})=n^{a_k}\prod_{r=0}^kR_r(n)^{d_{kr}^1}=n^{a_k}\prod_{r=0}^kR_r(n)^{\#\mathcal P^r(k)-\#\mathcal P^{2r+1}(k)}$$

Once again, a similar manipulation as in (1) gives the formula in the statement.
\end{proof}

Observe that the quantum group $H_n^+$ cannot be included into the above general theorem, and this for 2 reasons: first, because the orthogonal polynomial interpretation of the polynomials appearing in Theorem 6.3. fails, cf. the previous section, and second, because the epi interpretation of the exponents appearing in Theorem 6.3 seems to fail as well. 

\section{Concluding remarks}

We have seen in this paper that the Gram matrix determinants have a natural interpretation in the easy quantum group framework, developed in \cite{bsp}, \cite{ez1}, \cite{ez2}, \cite{ez3}. The known computations, that we partly extended, simplified, or rearranged in this paper, provide a complete set of formulae for the main examples of easy quantum groups.

Our conjecture is that these Gram determinants should have general decompositions of type $\det(G_{kn})=\prod_{\pi\in \mathcal P(k)}\varphi(\pi)$. More precisely, the situation here is as follows:
\begin{enumerate}
\item For $S_n,H_n,H_n^*$ the conjecture holds, with $\varphi(\pi)=n!/(n-|\pi|)!$.

\item For $O_n,B_n,O_n^*$ we have a decomposition result, but over Young diagrams.

\item For $O_n^+,B_n^+,S_n^+$ we have a decomposition result, but over the associated epi.
\end{enumerate}

The remaining problem is to find the correct surjective maps for (2,3), i.e. the correct surjections from diagrams/epi to partitions.  Of course, this question is not very clearly formulated. The main problem is probably to understand the behavior of the Gram matrix determinants in relation with the liberation operation $G_n\to G_n^+$. Indeed, we expect in this situation the contributions $\varphi$ to be related by a kind of induction/restriction procedure.

In addition to the concrete computations performed in this paper, let us mention that there are as well some quite heavy, abstract methods, that we haven't really tried yet. First, the inclusion $G_n\subset G_n^+$ gives rise to a planar algebra module in the sense of Jones \cite{jo2}, and our above ``liberation conjecture'' can be understood as saying that the Gram matrix combinatorics behaves well with respect to this planar module structure. And second, modulo the orthogonal polynomial issues discussed in the previous section, some useful tools should come from the analytic theory of the Bercovici-Pata bijection \cite{bpa}.

\end{document}